\documentclass[12pt,twoside]{amsart}
\usepackage{amssymb}
\usepackage{amscd}
\usepackage{amsmath}
\usepackage[all]{xy}

\title[Effective basepoint-free theorem]
{Effective basepoint-free theorem for semi-log canonical surfaces}
\author{Osamu Fujino}
\date{2017/1/18, version 0.10}
\address{Department of Mathematics, School of Science, Osaka University, 
Toyonaka, Osaka 560--0043, Japan}
\email{fujino@math.sci.osaka-u.ac.jp}

\newcommand{\Supp}[0]{\operatorname{Supp}}
\newcommand{\Exc}[0]{\operatorname{Exc}}
\newcommand{\mult}[0]{\operatorname{mult}}
\newcommand{\Sing}[0]{\operatorname{Sing}}
\newcommand{\Nlc}[0]{\operatorname{Nlc}}
\newcommand{\Nqlc}[0]{\operatorname{Nqlc}}
\newcommand{\Bs}[0]{\operatorname{Bs}}
\newtheorem{thm}{Theorem}[section]
\newtheorem{lem}[thm]{Lemma}
\newtheorem{cor}[thm]{Corollary}
\newtheorem{conj}[thm]{Conjecture}

\theoremstyle{definition}
\newtheorem{defn}[thm]{Definition}
\newtheorem{rem}[thm]{Remark}
\newtheorem*{ack}{Acknowledgments}
\newtheorem{say}[thm]{}
\newtheorem{step}{Step}
\newtheorem{case}{Case} 

\makeatletter
    
    \@addtoreset{equation}{section}
\makeatother
\begin{document}
\bibliographystyle{amsalpha+}

\maketitle

\begin{abstract}
This paper proposes 
a Fujita-type freeness conjecture for semi-log canonical pairs. 
We prove it for curves and surfaces by using the theory of 
quasi-log schemes and 
give some effective very ampleness results for 
stable surfaces and semi-log canonical Fano surfaces. 
We also prove an effective freeness for log surfaces.  
\end{abstract}

\section{Introduction}\label{f-sec1}

We will work over $\mathbb C$, the complex number field, throughout this paper.
Note that, by the Lefschetz principle,
all the results in 
this paper hold over any algebraically closed field $k$ of characteristic zero.

This paper proposes the following 
Fujita-type freeness conjecture for projective 
semi-log canonical pairs. 

\begin{conj}[Fujita-type freeness conjecture for 
semi-log canonical pairs]\label{f-conj1.1} 
Let $(X, \Delta)$ be an $n$-dimensional 
projective semi-log canonical pair and let $D$ be a Cartier divisor on $X$. 
We put $A=D-(K_X+\Delta)$. 
Assume that 
\begin{itemize}
\item[(1)] $(A^n\cdot {X_i})>n^n$ for 
every irreducible component $X_i$ of $X$, and 
\item[(2)] $(A^d\cdot W)\geq n^d$ for every $d$-dimensional 
irreducible subvariety $W$ of $X$ for $1\leq d\leq n-1$. 
\end{itemize} 
Then the complete linear system $|D|$ is basepoint-free. 
\end{conj}

By \cite[Corollary 3.5]{liu-san}, the complete linear system $|D|$ is basepoint-free 
if $A^n>\left(\frac{1}{2} n(n+1)\right)^n$ and 
$(A^d\cdot W)> 
\left(\frac{1}{2}n(n+1)\right)^d$ hold true in Conjecture 
\ref{f-conj1.1}, 
which is 
obviously a generalization of Anghern--Siu's effective 
freeness (see \cite{anghern-siu} and \cite{fujino-effective}). 

Of course, the above conjecture is a naive generalization of 
Fujita's celebrated conjecture: 

\begin{conj}[Fujita's freeness conjecture]\label{f-conj1.2}
Let $X$ be a smooth projective variety 
with $\dim X=n$ and let $H$ be an ample Cartier divisor on $X$. 
Then the complete linear system $|K_X+(n+1)H|$ is basepoint-free. 
\end{conj} 

The main theorem of this paper is: 

\begin{thm}[Main theorem, see Theorem \ref{f-thm2.1} and 
Theorem \ref{f-thm5.1}]\label{f-thm1.3}
Conjecture \ref{f-conj1.1} holds true in dimension one and two. 
\end{thm}

As a corollary of Theorem \ref{f-thm1.3}, 
we have: 

\begin{cor}[{cf.~\cite[Theorem 24]{lr}}]\label{f-cor1.4}
Let $(X, \Delta)$ be a stable surface such that $K_X+\Delta$ is 
$\mathbb Q$-Cartier. 
Let $I$ be the smallest positive integer such that $I(K_X+\Delta)$ 
is Cartier. Then $|mI(K_X+\Delta)|$ is basepoint-free 
and $3mI(K_X+\Delta)$ is very ample for every 
$m\geq 4$. 
If $I\geq 2$, then $|mI(K_X+\Delta)|$ is basepoint-free 
and $3mI(K_X+\Delta)$ is very ample 
for every $m\geq 3$. In particular, $12I(K_X+\Delta)$ is always 
very ample and $9I(K_X+\Delta)$ is very 
ample if $I\geq 2$. 
\end{cor}

Note that a {\em{stable pair}} $(X, \Delta)$ is a projective 
semi-log canonical pair $(X, \Delta)$ such that $K_X+\Delta$ is ample. 
A {\em{stable surface}} is a $2$-dimensional stable pair.  
We also have: 

\begin{cor}[Semi-log canonical Fano surfaces]\label{f-cor1.5} 
Let $(X, \Delta)$ be a projective semi-log canonical surface such that 
$-(K_X+\Delta)$ is an ample $\mathbb Q$-divisor. 
Let $I$ be the smallest positive integer such that 
$I(K_X+\Delta)$ is Cartier. 
Then $|-mI(K_X+\Delta)|$ is basepoint-free 
and $-3mI(K_X+\Delta)$ is very ample for every $m\geq 2$. 
In particular, $-6I(K_X+\Delta)$ is very ample.  
\end{cor}

For log surfaces (see \cite{fujino-surface}), 
the following theorem is a reasonable formulation of 
the Reider-type freeness theorem. 
For a related topic, see \cite{kawachi}.  

\begin{thm}[Effective freeness for log surfaces]\label{f-thm1.6}
Let $(X, \Delta)$ be a complete irreducible log surface and let $D$ be 
a Cartier divisor on $X$. 
We put $A=D-(K_X+\Delta)$. Assume that 
$A$ is nef, $A^2>4$ and $A\cdot C\geq 2$ for every curve $C$ on $X$ such that 
$x\in C$. 
Then $\mathcal O_X(D)$ has a global section not vanishing at $x$. 
\end{thm}

We know that the theory of log surfaces initiated in \cite{fujino-surface} 
now holds in characteristic $p>0$ 
(see \cite{fujino-tanaka}, \cite{tanaka-nagoya}, and \cite{tanaka-vanishing}). 
Therefore, it is natural to propose: 

\begin{conj}\label{f-conj1.7}
Theorem \ref{f-thm1.6} holds in characteristic $p>0$. 
\end{conj}

Note that the original form of Fujita's freeness conjecture (see Conjecture \ref{f-conj1.2}) 
is still open for surfaces in characteristic $p>0$. 

The standard approach to 
the Fujita-type freeness conjectures is based on 
the Kawamata--Viehweg vanishing theorem 
(see \cite{ein-lazarsfeld}). 
However, we can not directly apply the Kawamata--Viehweg 
vanishing theorem to log canonical pairs and semi-log canonical pairs. 
Therefore,, we will use the theory of quasi-log schemes (see \cite{fujino-slc}, 
\cite{fujino-reid-fukuda}, \cite{fujino-foundation}, and so on). 

We summarize the contents of this paper. 
In Section \ref{f-sec2}, 
we prove Conjecture \ref{f-conj1.1} for semi-log canonical 
curves using the 
vanishing theorem obtained in \cite{fujino-slc}. 
This section may help the reader to understand more complicated 
arguments in the subsequent sections. 
In Section \ref{f-sec3}, we collect some basic definitions. 
In Section \ref{f-sec4}, 
we quickly recall the theory of quasi-log schemes. 
Section \ref{f-sec5} is the main part of this paper. 
In this section, we prove Conjecture \ref{f-conj1.1} for semi-log canonical 
surfaces. Section \ref{f-sec6} is devoted to the 
proof of  Theorem \ref{f-thm1.6}, which is an effective freeness for 
log surfaces. 
In Section \ref{f-sec7}, 
which is independent of the other sections, 
we prove 
an effective very ampleness lemma. 

\begin{ack}
The author was partially supported 
by JSPS KAKENHI Grant Numbers JP2468002, JP16H03925, JP16H06337. 
He would like to thank Professor J\'anos Koll\'ar for answering his 
question. Finally, he thanks the referee very much for many useful comments and 
suggestions. 
\end{ack}

For the standard 
notations and conventions of the 
minimal model program, see \cite{fujino-fundamental} and \cite{fujino-foundation}. 
For the details of semi-log canonical pairs, see \cite{fujino-slc}. 
In this paper, a {\em{scheme}} means a separated scheme of finite type over 
$\mathbb C$ and a {\em{variety}} means a reduced scheme. 

\section{Semi-log canonical curves}\label{f-sec2}
In this section, we prove Conjecture \ref{f-conj1.1} in dimension one 
based on \cite{fujino-slc}. 
This section will help the reader to understand the subsequent sections. 

\begin{thm}\label{f-thm2.1}
Let $(X, \Delta)$ be a projective semi-log canonical 
curve and let $D$ be a Cartier divisor on $X$. 
We put $A=D-(K_X+\Delta)$. 
Assume that $(A\cdot X_i)>1$ for every irreducible component $X_i$ of $X$. 
Then the complete linear system $|D|$ is basepoint-free. 
\end{thm}

If $(X, \Delta)$ is log canonical, that is, $X$ is normal, in Theorem \ref{f-thm2.1}, 
then the statement is obvious. 
However, Theorem \ref{f-thm2.1} seems to be nontrivial when $X$ is not normal. 

\begin{proof}[Proof of Theorem \ref{f-thm2.1}]
We will see that 
the restriction map 
\begin{equation}\label{f-eq2.1}
H^0(X, \mathcal O_X(D))\to \mathcal O_X(D)\otimes \mathbb C(P)
\end{equation} 
is surjective for every $P\in X$. 
Of course, it is sufficient to prove that $H^1(X, \mathcal I_P\otimes \mathcal O_X(D))=0$, 
where $\mathcal I_P$ is the defining ideal sheaf of $P$ on $X$. 
If $P$ is a zero-dimensional 
semi-log canonical center of $(X, \Delta)$, then 
we know that 
$H^1(X, \mathcal I_P\otimes \mathcal O_X(D))=0$ by \cite[Theorem 1.11]{fujino-slc}. 
Therefore, we may assume that $P$ is not a 
zero-dimensional semi-log canonical center of $(X, \Delta)$. 
Thus, we see that $X$ is normal, that is, smooth, at $P$ 
(see, for example, \cite[Corollary 3.5]{fujino-slc}). 
We put 
\begin{equation}
c=1-\mult_P \Delta. 
\end{equation} 
Then we have $0<c\leq 1$. 
We consider $(X, \Delta+cP)$. 
Then $(X, \Delta+cP)$ is semi-log canonical and $P$ is a zero-dimensional 
semi-log canonical center of $(X, \Delta+cP)$. 
Since 
\begin{equation}
\left( (D-(K_X+\Delta+cP))\cdot X_i\right)>0
\end{equation} 
for every irreducible component $X_i$ of $X$ by the assumption that 
$(A\cdot X_i)>1$ and the fact that $c\leq 1$, 
we obtain that 
$H^1(X, \mathcal I_P\otimes \mathcal O_X(D))=0$ 
(see \cite[Theorem 1.11]{fujino-slc}). 
Therefore, we see that $H^1(X, \mathcal I_P\otimes \mathcal O_X(D))=0$ for 
every $P\in X$. 
Thus, we have the desired surjection \eqref{f-eq2.1}. 
\end{proof}

The above proof of Theorem \ref{f-thm2.1} heavily depends on the vanishing theorem 
for semi-log canonical pairs (see \cite[Theorem 1.11]{fujino-slc}), 
which follows from 
the theory of quasi-log schemes based on 
the theory of mixed Hodge structures on cohomology with compact support. 
For the details, see \cite{fujino-slc} and \cite{fujino-foundation}. 
In dimension two, we will directly use the framework of quasi-log 
schemes. 
Therefore, it is much more difficult than the proof of Theorem \ref{f-thm2.1}. 

\section{Preliminaries}\label{f-sec3}

In this section, we collect some basic definitions. 

\begin{say}[Operations for $\mathbb R$-divisors]\label{f-say3.1}
Let $D$ be an $\mathbb R$-divisor 
on an equidimensional variety $X$, that is, 
$D$ is a finite formal $\mathbb R$-linear combination 
\begin{equation}
D=\sum _i d_i D_i
\end{equation} of irreducible 
reduced subschemes $D_i$ of codimension one, where $D_i\ne D_j$ for $i\ne j$. 
We define the {\em{round-up}} 
$\lceil D\rceil =\sum _i \lceil d_i \rceil D_i$ (resp.~{\em{round-down}} 
$\lfloor D\rfloor =\sum _i \lfloor d_i \rfloor D_i$), where for 
every real number $x$, $\lceil x\rceil$ (resp.~$\lfloor x\rfloor$) is the integer 
defined by $x\leq \lceil x\rceil <x+1$ 
(resp.~$x-1<\lfloor x\rfloor \leq x$). 
We put 
\begin{equation}
D^{<1}=\sum _{d_i<1}d_i D_i \quad {\text{and}}\quad 
D^{>1}=\sum _{d_i>1}d_i D_i. 
\end{equation}
We call $D$ a {\em{boundary}} (resp.~{\em{subboundary}}) $\mathbb R$-divisor if 
$0\leq d_i\leq 1$ (resp.~$d_i\leq 1$) for every $i$. 
\end{say}

\begin{say}[Singularities of pairs]\label{f-say3.2}
Let $X$ be a normal variety and let $\Delta$ be an 
$\mathbb R$-divisor on $X$ 
such that $K_X+\Delta$ is $\mathbb R$-Cartier. 
Let $f:Y\to X$ be 
a resolution such that $\Exc(f)\cup f^{-1}_*\Delta$, 
where $\Exc (f)$ is the exceptional locus of $f$ 
and $f^{-1}_*\Delta$ is 
the strict transform of $\Delta$ on $Y$,  
has a simple normal crossing support. We can 
write 
\begin{equation}\label{f-eq3.1}
K_Y=f^*(K_X+\Delta)+\sum _i a_i E_i. 
\end{equation}
We say that $(X, \Delta)$ 
is {\em{sub log canonical}} ({\em{sub lc}}, for short) if $a_i\geq -1$ for every $i$. 
We usually write $a_i= a(E_i, X, \Delta)$
and call it the {\em{discrepancy coefficient}} of 
$E_i$ with respect to $(X, \Delta)$. 
Note that we can define $a(E, X, \Delta)$ for every prime divisor 
$E$ {\em{over}} $X$. 
If $(X, \Delta)$ is sub log canonical and $\Delta$ is effective, then 
$(X, \Delta)$ is called {\em{log canonical}} ({\em{lc}}, for short). 

It is well-known that there is the largest Zariski open subset $U$ 
of $X$ such that  
$(U, \Delta|_U)$ is sub log canonical (see, for example, 
\cite[Lemma 2.3.10]{fujino-foundation}). 
If there exist a resolution $f:Y\to X$ and a divisor $E$ on $Y$ such 
that $a(E, X, \Delta)=-1$ and $f(E)\cap U\ne \emptyset$, then $f(E)$ is called a 
{\em{log canonical center}} (an {\em{lc center}}, for short) with respect to $(X, \Delta)$. 
A closed subset $C$ of $X$ is called a {\em{log canonical stratum}} 
(an {\em{lc stratum}}, for short) of $(X, \Delta)$ if and only if 
$C$ is a log canonical center of $(X, \Delta)$ or $C$ is 
an irreducible component of $X$. 
We note that the {\em{non-lc locus}} of $(X, \Delta)$, which is 
denoted by $\Nlc(X, \Delta)$, is $X\setminus U$. 

Let $X$ be a normal variety and let $\Delta$ be an effective $\mathbb R$-divisor on 
$X$ such that $K_X+\Delta$ is $\mathbb R$-Cartier. 
If $a(E, X, \Delta)>-1$ for every divisor $E$ over $X$, 
then $(X, \Delta)$ is called {\em{klt}}. If 
$a(E, X, \Delta)>-1$ for every exceptional divisor $E$ over $X$, 
then $(X, \Delta)$ is called {\em{plt}}. 
\end{say}

Let us recall the definitions around {\em{semi-log canonical pairs}}. 

\begin{say}[Semi-log canonical pairs]\label{f-say3.3}
Let $X$ be an 
equidimensional variety that 
satisfies Serre's $S_2$ condition and 
is normal crossing in codimension one. 
Let $\Delta$ be an effective $\mathbb R$-divisor 
whose support does not contain any irreducible components 
of the conductor of $X$. 
The pair $(X, \Delta)$ is called a {\em{semi-log canonical pair}} (an {\em{slc pair}}, 
for short) 
if 
\begin{itemize}
\item[(1)] $K_X+\Delta$ is $\mathbb R$-Cartier, and 
\item[(2)] $(X^\nu, \Theta)$ is log canonical, 
where $\nu:X^\nu\to X$ is the normalization and $K_{X^\nu}+\Theta=
\nu^*(K_X+\Delta)$, that is, $\Theta$ is the sum of the 
inverse images of $\Delta$ and the conductor of $X$. 
\end{itemize}

Let $(X, \Delta)$ be a semi-log canonical pair and let $\nu:X^\nu\to X$ be 
the normalization. 
We set  
\begin{equation}
K_{X^\nu}+\Theta=\nu^*(K_X+\Delta)
\end{equation} 
as above. 
A closed subvariety $W$ of $X$ is called a {\em{semi-log canonical center}} 
(an {\em{slc center}}, for short) {\em{with 
respect to $(X, \Delta)$}} if there exist a resolution of singularities $f: Y\to X^\nu$ and 
a prime divisor $E$ on $Y$ such that 
the discrepancy coefficient $a(E, X^\nu, \Theta)=-1$ and $\nu\circ f(E)=W$. 
A closed subvariety $W$ of $X$ is called a {\em{semi-log canonical stratum}} 
({\em{slc stratum}}, for short) of 
the pair $(X, \Delta)$ if 
$W$ is a semi-log canonical center with respect to $(X, \Delta)$ or $W$ is an 
irreducible component of $X$. 
\end{say}

We close this section with the notion of {\em{log surfaces}} (see \cite{fujino-surface}). 

\begin{say}[Log surfaces]\label{f-say3.4}
Let $X$ be a normal surface and let $\Delta$ be a boundary $\mathbb R$-divisor on $X$. 
Assume that $K_X+\Delta$ is $\mathbb R$-Cartier. 
Then the pair $(X, \Delta)$ is called a {\em{log surface}}. 
A log surface $(X, \Delta)$ is not always assumed to be log canonical. 

In \cite{fujino-surface}, 
we establish the minimal model program for log surfaces in full generality 
under the assumption that $X$ is $\mathbb Q$-factorial or $(X, \Delta)$ has only 
log canonical singularities. 
For the theory of log surfaces in characteristic $p>0$, 
see \cite{fujino-tanaka}, \cite{tanaka-nagoya}, and \cite{tanaka-vanishing}.  
\end{say}

\section{On quasi-log structures}\label{f-sec4}
Let us quickly recall the definitions of {\em{globally embedded simple 
normal crossing pairs}} and {\em{quasi-log schemes}} for 
the reader's convenience. 
For the details, see, for example, \cite{fujino-pull} 
and \cite[Chapter 5 and Chapter 6]{fujino-foundation}. 

\begin{defn}[Globally embedded simple normal crossing 
pairs]\label{f-def4.1} 
Let $Y$ be a simple normal crossing divisor 
on a smooth 
variety $M$ and let $D$ be an $\mathbb R$-divisor 
on $M$ such that 
$\Supp (D+Y)$ is a simple normal crossing divisor on $M$ and that 
$D$ and $Y$ have no common irreducible components. 
We put $B_Y=D|_Y$ and consider the pair $(Y, B_Y)$. 
We call $(Y, B_Y)$ a {\em{globally embedded simple normal 
crossing pair}} and $M$ the {\em{ambient space}} of $(Y, B_Y)$. 
A {\em{stratum}} of $(Y, B_Y)$ is 
the $\nu$-image of a log canonical stratum of $(Y^\nu, \Theta)$ 
where $\nu:Y^\nu\to Y$ is the normalization and $K_{Y^\nu}+\Theta
=\nu^*(K_Y+B_Y)$, that is, $\Theta$ is the sum of 
the inverse images of $B_Y$ and the singular locus of $Y$. 
\end{defn}

In this paper, we adopt the following definition of 
quasi-log schemes. 

\begin{defn}[Quasi-log schemes]\label{f-def4.2}
A {\em{quasi-log scheme}} is a scheme $X$ endowed with an 
$\mathbb R$-Cartier divisor 
(or $\mathbb R$-line bundle) 
$\omega$ on $X$, a proper closed subscheme 
$X_{-\infty}\subset X$, and a finite collection $\{C\}$ of reduced 
and irreducible subschemes of $X$ such that there is a 
proper morphism $f:(Y, B_Y)\to X$ from a globally 
embedded simple 
normal crossing pair satisfying the following properties: 
\begin{itemize}
\item[(1)] $f^*\omega\sim_{\mathbb R}K_Y+B_Y$. 
\item[(2)] The natural map 
$\mathcal O_X
\to f_*\mathcal O_Y(\lceil -(B_Y^{<1})\rceil)$ 
induces an isomorphism 
$$
\mathcal I_{X_{-\infty}}\overset{\simeq}{\longrightarrow} f_*\mathcal O_Y(\lceil 
-(B_Y^{<1})\rceil-\lfloor B_Y^{>1}\rfloor),  
$$ 
where $\mathcal I_{X_{-\infty}}$ is the defining ideal sheaf of 
$X_{-\infty}$. 
\item[(3)] The collection of subvarieties $\{C\}$ coincides with the image 
of $(Y, B_Y)$-strata that are not included in $X_{-\infty}$. 
\end{itemize}
We simply write $[X, \omega]$ to denote 
the above data 
$$
\bigl(X, \omega, f:(Y, B_Y)\to X\bigr)
$$ 
if there is no risk of confusion. 
Note that a quasi-log scheme $X$ is the union of $\{C\}$ and $X_{-\infty}$. 
We also note that $\omega$ is called the {\em{quasi-log canonical class}} 
of $[X, \omega]$, which is defined up to $\mathbb R$-linear equivalence.  
We sometimes simply say that 
$[X, \omega]$ is a {\em{quasi-log pair}}. 
The subvarieties $C$ 
are called the {\em{qlc strata}} of $[X, \omega]$, 
$X_{-\infty}$ is called the {\em{non-qlc locus}} 
of $[X, \omega]$, and $f:(Y, B_Y)\to X$ is 
called a {\em{quasi-log resolution}} 
of $[X, \omega]$. 
We sometimes use $\Nqlc(X, \omega)$ to denote 
$X_{-\infty}$. A closed subvariety $C$ of $X$ is called a {\em{qlc center}} 
of $[X, \omega]$ if $C$ is a qlc stratum of $[X, \omega]$ which is not 
an irreducible component of $X$. 

Let $[X, \omega]$ be a quasi-log scheme. 
Assume that $X_{-\infty}=\emptyset$. 
Then we sometimes simply say that $[X, \omega]$ is 
a {\em{qlc pair}} or 
$[X, \omega]$ is a quasi-log scheme with only {\em{quasi-log canonical 
singularities}}. 
\end{defn}

\begin{defn}[Nef and log big divisors 
for quasi-log schemes]\label{f-def4.3} 
Let $L$ be an $\mathbb R$-Cartier divisor (or 
$\mathbb R$-line bundle) on a quasi-log pair $[X, \omega]$ and 
let $\pi:X\to S$ be a proper morphism between schemes. 
Then $L$ is {\em{nef and log big over $S$ with 
respect to $[X, \omega]$}} if $L$ is 
$\pi$-nef and $L|_C$ is $\pi$-big for every 
qlc stratum $C$ of $[X, \omega]$. 
\end{defn}

The following theorem is a key result for the theory of quasi-log schemes. 

\begin{thm}[Adjunction and vanishing theorem]\label{f-thm4.4}
Let $[X, \omega]$ be a quasi-log scheme and let $X'$ be the union of 
$X_{-\infty}$ with a {\em{(}}possibly empty{\em{)}} union of some 
qlc strata of $[X, \omega]$. Then 
we have the following properties. 
\begin{itemize}
\item[(i)] Assume that $X'\ne X_{-\infty}$. Then 
$X'$ is a quasi-log scheme with $\omega'=\omega|_{X'}$ and 
$X'_{-\infty}=X_{-\infty}$. Moreover, the qlc strata of $[X', \omega']$ are 
exactly the qlc strata of $[X, \omega]$ that are included in $X'$. 
\item[(ii)] Assume that $\pi:X\to S$ is a proper morphism between schemes. 
Let $L$ be a Cartier divisor on $X$ such that 
$L-\omega$ is nef and log big over $S$ with respect to $[X, \omega]$. 
Then $R^i\pi_*(\mathcal I_{X'}\otimes \mathcal O_X(L))=0$ for every $i>0$, 
where $\mathcal I_{X'}$ is the defining ideal sheaf of $X'$ on $X$. 
\end{itemize}
\end{thm}

For the proof of Theorem \ref{f-thm4.4}, see, 
for example, 
\cite[Theorem 3.8]{fujino-reid-fukuda} and \cite[Section 6.3]{fujino-foundation}. 
We can slightly generalize Theorem \ref{f-thm4.4} (ii) as follows. 

\begin{thm}\label{f-thm4.5}
Let $[X, \omega]$, $X'$, and $\pi:X\to S$ be as in Theorem \ref{f-thm4.4}. 
Let $L$ be a Cartier divisor on $X$ such that 
$L-\omega$ is 
nef over $S$ and that 
$(L-\omega)|_W$ is big over $S$ for any qlc stratum $W$ of $[X, \omega]$ which 
is not contained in $X'$. 
Then $R^i\pi_*(\mathcal I_{X'}\otimes \mathcal O_X(L))=0$ for every 
$i>0$, where $\mathcal I_{X'}$ is the defining ideal 
sheaf of $X'$ on $X$. 
\end{thm}

Theorem \ref{f-thm4.5} is obvious by the proof of Theorem \ref{f-thm4.4}. 
For a related topic, see \cite[Remark 5.2]{fujino-slc}. 
Theorem \ref{f-thm4.5} will play a crucial role in the proof of Theorem \ref{f-thm1.6} 
in Section \ref{f-sec6}. 
 
Finally, 
we prepare a useful lemma, which is new, for the proof of Theorem \ref{f-thm1.3}. 

\begin{lem}\label{f-lem4.6}
Let $[X, \omega]$ be a qlc pair such that 
$X$ is irreducible. 
Let $E$ be an effective $\mathbb R$-Cartier divisor on $X$. 
This means that 
$$
E=\sum _{i=1}^k e_i E_i
$$ 
where $E_i$ is an effective Cartier divisor on $X$ and $e_i$ is a positive 
real number for every $i$. 
Then we can give a quasi-log structure to $[X, \omega+E]$, 
which coincides with the original quasi-log structure of $[X, \omega]$ outside 
$\Supp E$. 
\end{lem}
For the details of the quasi-log structure of $[X, \omega+E]$, see 
the construction in the proof below. 

\begin{proof}
Let $f:(Z, \Delta_Z)\to [X, \omega]$ be a quasi-log resolution, where 
$(Z, \Delta_Z)$ is a globally embedded simple normal crossing pair. 
By taking some suitable blow-ups, we may assume that 
the union of all strata of $(Z, \Delta_Z)$ mapped 
to $\Supp E$, which is denoted by $Z''$, is a union of 
some irreducible components of $Z$ (see \cite[Proposition 4.1]{fujino-pull} 
and \cite[Section 6.3]{fujino-foundation}). 
We put $Z'=Z-Z''$ and 
$K_{Z'}+\Delta_{Z'}=(K_Z+\Delta_Z)|_{Z'}$. 
We may further assume that 
$(Z', \Delta_{Z'}+{f'}^*E)$ is a globally embedded simple normal crossing pair, 
where $f'=f|_{Z'}: Z'\to X$. 
By construction, we have a natural inclusion 
\begin{equation}
\mathcal O_{Z'} (\lceil -(\Delta_{Z'}+{f'}^*E)^{<1}\rceil -
\lfloor (\Delta_{Z'}+{f'}^*E)^{>1}\rfloor)\subset \mathcal 
O_Z(\lceil -\Delta_Z^{<1}\rceil). 
\end{equation}
This is because 
\begin{equation}
-\lfloor (\Delta_{Z'}+f'^*E)^{>1}\rfloor \leq -Z''|_{Z'} 
\end{equation}
and 
\begin{equation}
\mathcal O_{Z'}(-Z''|_{Z'})\subset \mathcal O_Z. 
\end{equation}
Thus, we have 
\begin{equation}
f'_*\mathcal O_{Z'} (\lceil -(\Delta_{Z'}+{f'}^*E)^{<1}\rceil -
\lfloor (\Delta_{Z'}+{f'}^*E)^{>1}\rfloor)
\subset f_*\mathcal 
O_Z(\lceil -\Delta_Z^{<1}\rceil)\simeq \mathcal O_X. 
\end{equation}
By putting 
\begin{equation}
\mathcal I_{X_{-\infty}}= 
{f'}_*\mathcal O_{Z'} (\lceil -(\Delta_{Z'}+{f'}^*E)^{<1}\rceil -
\lfloor (\Delta_{Z'}+{f'}^*E)^{>1}\rfloor), 
\end{equation} 
$f': (Z', \Delta_{Z'}+{f'}^*E)\to [X, \omega+E]$ gives a quasi-log structure to 
$[X, \omega+E]$. 
By construction, it coincides with 
the original quasi-log structure of $[X, \omega]$ outside $\Supp E$. 
\end{proof}

\section{Semi-log canonical surfaces}\label{f-sec5}

In this section, we prove Conjecture \ref{f-conj1.1} for surfaces. 

\begin{thm}\label{f-thm5.1}
Let $(X, \Delta)$ be a projective semi-log canonical 
surface and let $D$ be a Cartier divisor on $X$. 
We put $A=D-(K_X+\Delta)$.  
Assume that $(A^2\cdot {X_i})>4$ for every irreducible component 
$X_i$ of $X$ and 
that $A\cdot C\geq 2$ for every curve $C$ on $X$. 
Then the complete linear system $|D|$ is basepoint-free. 
\end{thm}

\begin{rem}\label{f-rem5.2}
By assumption and Nakai's ampleness criterion for $\mathbb R$-divisors 
(see \cite{campana-peternell}), $A$ is ample in Theorem \ref{f-thm5.1}. 
However, we do not use the ampleness of $A$ in the proof of Theorem \ref{f-thm5.1}. 
\end{rem}

Our proof of Theorem \ref{f-thm5.1} uses the theory of quasi-log schemes. 

\begin{proof} 
We will prove 
that the restriction map 
$$
H^0(X, \mathcal O_X(D))\to \mathcal O_X(D)\otimes \mathbb C(P)
$$ 
is surjective for every $P\in X$. 
\begin{step}[Quasi-log structure]\label{step1}
By \cite[Theorem 1.2]{fujino-slc}, we can 
take a quasi-log resolution $f:(Z, \Delta_Z)\to [X, K_X+\Delta]$. 
Precisely speaking, 
$(Z, \Delta_Z)$ is a globally embedded simple normal 
crossing pair such that 
$\Delta_Z$ is a subboundary $\mathbb R$-divisor on $Z$ with the following properties. 
\begin{itemize}
\item[(i)] $K_Z+\Delta_Z\sim _{\mathbb R} f^*(K_X+\Delta)$. 
\item[(ii)] the natural map $\mathcal O_X\to f_*\mathcal O_Z(\lceil -\Delta^{<1}_Z\rceil)$ 
is an isomorphism. 
\item[(iii)] $\dim Z=2$. 
\item[(iv)] $W$ is a semi-log canonical stratum of $(X, \Delta)$ if and only if 
$W=f(S)$ for some stratum $S$ of $(Z, \Delta_Z)$. 
\end{itemize} 
It is worth mentioning that $f:Z\to X$ is not necessarily birational. 
This step is nothing but \cite[Theorem 1.2]{fujino-slc}. 
\end{step}
\begin{step}
Assume that $P$ is a zero-dimensional 
semi-log canonical center of $(X, \Delta)$. 
Then $H^i(X, \mathcal I_P\otimes \mathcal O_X(D))=0$ for every $i>0$, 
where $\mathcal I_P$ is the defining ideal 
sheaf of $P$ on $X$ (see \cite[Theorem 1.11]{fujino-slc} and 
Theorem \ref{f-thm4.4}). 
Therefore, the restriction map 
$$
H^0(X, \mathcal O_X(D))\to \mathcal O_X(D)\otimes \mathbb C(P)
$$ 
is surjective. 
\end{step}
From now on, we may assume that $P$ is not a zero-dimensional 
semi-log canonical center of $(X, \Delta)$. 
\begin{step}\label{step3} 
Assume that there exists a one-dimensional 
semi-log canonical center $W$ of $(X, \Delta)$ such that 
$P\in W$. 
Since $P$ is not a zero-dimensional 
semi-log canonical center of $(X, \Delta)$, $W$ is normal, that is, 
smooth, at $P$ by \cite[Corollary 3.5]{fujino-slc}. 
By adjunction (see Theorem \ref{f-thm4.4}), 
$[W, (K_X+\Delta)|_W]$ has a quasi-log structure with only quasi-log 
canonical singularities 
induced by the quasi-log 
structure $f:(Z, \Delta_Z)\to [X, K_X+\Delta]$ constructed in Step \ref{step1}. 
Let $g:(Z', \Delta_{Z'})\to [W, (K_X+\Delta)|_W]$ be the induced quasi-log resolution. 
We put 
\begin{equation}
c=\underset{t\geq 0}{\sup}\left\{t \left|
\begin{array}{l}  {\text{the normalization of 
$(Z', \Delta_{Z'}+tg^*P)$ is}}\\
{\text{sub log canonical.}} 
\end{array}\right. \right\}. 
\end{equation} 
Then, by \cite[Lemma 3.16]{fujino-reid-fukuda}, we obtain that $0<c<2$. 
Note that $P$ is a Cartier divisor on $W$. 
Let us consider $g:(Z', \Delta_{Z'}+cg^*P)\to [W, (K_X+\Delta)|_W+cP]$, 
which defines a quasi-log structure. 
Then, by construction, $P$ is a qlc center of $[W, (K_X+\Delta)|_W+cP]$. 
Moreover, 
we see that 
\begin{equation}
(D|_W-((K_X+\Delta)|_W+cP))=(A\cdot W)-c>0
\end{equation} by assumption. 
Therefore, we obtain that 
\begin{equation}
H^i(W, \mathcal I_P\otimes \mathcal O_W(D))=0
\end{equation} 
for every $i>0$ by Theorem \ref{f-thm4.4}, where 
$\mathcal I_P$ is the defining ideal sheaf of $P$ on $W$. 
Thus, the restriction map 
\begin{equation}\label{eq54}
H^0(W, \mathcal O_W(D))\to \mathcal O_W(D)\otimes \mathbb C(P)
\end{equation} 
is surjective. On the other hand, 
by Theorem \ref{f-thm4.4} again, 
we have that 
\begin{equation}
H^i(X, \mathcal I_W\otimes \mathcal O_X(D))=0
\end{equation} for 
every $i>0$, 
where $\mathcal  I_W$ is the defining ideal sheaf of $W$ on $X$. 
This implies that 
the restriction map 
\begin{equation}\label{eq56}
H^0(X, \mathcal O_X(D))\to H^0(W, \mathcal O_W(D))
\end{equation} 
is surjective. 
By combining \eqref{eq54} with \eqref{eq56}, 
the desired restriction map 
\begin{equation}
H^0(X, \mathcal O_X(D))\to \mathcal O_X(D)\otimes \mathbb C(P)
\end{equation} 
is surjective. 
\end{step}
Therefore, from now on, we may assume that no one-dimensional 
semi-log canonical centers of $(X, \Delta)$ contain $P$. 
\begin{step}\label{step4}
In this step, we assume that $P$ is a smooth point of $X$. 
Let $X_0$ be the unique irreducible component of $X$ containing $P$. 
By adjunction (see Theorem \ref{f-thm4.4}), 
$[X_0, (K_X+\Delta)|_{X_0}]$ has a quasi-log structure with only quasi-log 
canonical singularities 
induced by 
the quasi-log structure $f:(Z, \Delta_Z)\to [X, K_X+\Delta]$ constructed in Step \ref{step1}. 
By Theorem \ref{f-thm4.4}, 
\begin{equation}
H^i(X, \mathcal I_{X_0}\otimes \mathcal O_X(D))=0
\end{equation} 
for every $i>0$, where $\mathcal I_{X_0}$ is the defining ideal sheaf of 
$X_0$ on $X$. 
Therefore, the restriction map 
\begin{equation}\label{eq59}
H^0(X, \mathcal O_X(D))\to H^0(X_0, \mathcal O_{X_0}(D))
\end{equation} 
is surjective. Thus, it is sufficient to prove that the natural restriction map 
\begin{equation}
H^0(X_0, \mathcal O_{X_0}(D))\to \mathcal O_{X_0}(D)\otimes \mathbb C(P)
\end{equation}  
is surjective. 
We put $A_0=A|_{X_0}$. 
Since $A_0^2>4$, we can find an effective $\mathbb R$-Cartier divisor 
$B$ on $X_0$ such that 
$\mult_P B>2$ and that 
$B\sim _{\mathbb R} A_0$. 
We put $U=X_0\setminus \Sing X_0$ and 
define  
\begin{equation}
c=\max\{ t\geq 0\, |\, (U, \Delta|_U +tB|_U) \ 
\text{is log canonical at $P$}\}. 
\end{equation} 
Then we obtain that $0<c<1$ since $\mult _P B>2$. 
By Lemma \ref{f-lem4.6}, 
we have a quasi-log structure on $[X_0, (K_X+\Delta)|_{X_0} +cB]$. 
By construction, there is a qlc center $W$ of $[X_0, (K_X+\Delta)|_{X_0}+cB]$ passing 
through $P$. 
Let $X'$ be the union of the non-qlc locus 
of $[X_0, (K_X+\Delta)|_{X_0}+cB]$ and the minimal qlc center 
$W_0$ of $[X_0, (K_X+\Delta)|_{X_0}+cB]$ passing through $P$. 
Note that $D|_{X_0}-((K_X+\Delta)|_{X_0}+cB)\sim _{\mathbb R} (1-c)A_0$. 
Then, by Theorem \ref{f-thm4.4}, 
\begin{equation}\label{eq512}
H^i(X_0, \mathcal I_{X'}\otimes \mathcal O_{X_0}(D))=0
\end{equation}
for every $i>0$, where $\mathcal I_{X'}$ is the defining ideal sheaf of $X'$ on $X_0$. 
\begin{case}\label{case1}
If $\dim W_0=0$, 
then $P$ is isolated in $\Supp \mathcal O_{X_0}/ \mathcal I_{X'}$. 
Therefore, the restriction map 
\begin{equation}\label{eq513}
H^0(X_0, \mathcal O_{X_0}(D))\to \mathcal O_{X_0}(D)\otimes \mathbb C(P)
\end{equation} 
is surjective. 
\end{case}
\begin{case}\label{case2}
If $\dim W_0=1$, then let us consider the quasi-log structure of $[X', ((K_X+
\Delta)|_{X_0}+cB)|_{X'}]$ induced by 
the quasi-log structure of 
$[X_0, (K_X+\Delta)|_{X_0}+cB]$ constructed above by Lemma \ref{f-lem4.6} 
(see Theorem \ref{f-thm4.4} (i)). 
From now on, we will see that we can take $0<c'\leq 1$ such that 
$P$ is a zero-dimensional 
qlc center of $[X', ((K_X+\Delta)|_{X_0}+cB)|_{X'}+c'P]$ as in Step \ref{step3}. 
By assumption, 
$(X, \Delta+cB)$ is plt in a neighborhood of $P$. 
We put $\mult _PB=2+a$ with $a>0$. 
We write $\Delta+cB=L+\Delta'$, where 
$L=W_0$ is the unique one-dimensional 
log canonical center of $(X, \Delta)$ passing through $P$ and $\Delta'=\Delta+cB-L$. 
We put $\mult _P(\Delta+cB)=1+\delta$ with $\delta \geq 0$, equivalently, 
$\delta=\mult _P \Delta' \geq 0$. 
Note that 
\begin{equation}
1+\delta= \mult _P (\Delta+cB)=\mult _P\Delta+ c(2+a). 
\end{equation} 
Therefore, we have 
\begin{equation}
c=\frac{1+\delta-\alpha}{2+a}, 
\end{equation} 
where $\alpha=\mult _P\Delta \geq 0$. 
We also note that 
\begin{equation}
\delta\leq \mult _P (\Delta' |_L)<1. 
\end{equation} 
Then, we can choose $c'= 1-\mult _P (\Delta'|_L)$. 
This is because $(X, \Delta+cB+c' H)$ is log canonical 
in a neighborhood of $P$ but is not plt at $P$, 
where $H$ is a general smooth curve passing through $P$. 

In this situation, we have 
\begin{equation}\label{eq517}
\begin{split}
&\deg (D|_L -(K_X+\Delta+cB)|_L -c'P) \\
&\geq \left( 1-\frac{1+\delta-\alpha}{2+a} \right) \cdot 2 -(1-\delta)\\ 
& = \frac{1}{2+a} ((2+a-1-\delta+\alpha) \cdot 2 -(2+a)(1-\delta)) \\ 
&= \frac{1}{2+a}(a+2\alpha +a\delta) \\ 
& \geq \frac{a}{2+a} >0. 
\end{split} 
\end{equation}
Thus, by Theorem \ref{f-thm4.4}, 
\begin{equation}
H^i(X', \mathcal I_{X''}\otimes \mathcal O_{X'}(D))=0
\end{equation}  
for every $i>0$, where $X''$ is the union of the non-qlc locus of $[X', 
((K_X+\Delta)|_{X_0} +cB)|_{X'}+c'P]$ and $P$, and 
$\mathcal I_{X''}$ is the defining ideal sheaf of $X''$ on $X'$. 
Thus, we have that 
\begin{equation}\label{eq519}
H^0(X', \mathcal O_{X'} (D))\to \mathcal O_{X'}(D)\otimes 
\mathcal O_{X'} / \mathcal I_{X''}
\end{equation} 
is surjective. 
Note that $P$ is isolated in $\Supp \mathcal O_{X'}/ \mathcal I_{X''}$. 
Therefore, we obtain surjections  
\begin{equation}
H^0(X, \mathcal O_X(D))
\twoheadrightarrow 
H^0(X_0, \mathcal O_{X_0} (D))
\twoheadrightarrow 
H^0(X', \mathcal O_{X'}(D)) \twoheadrightarrow 
\mathcal O_{X'} (D)\otimes \mathbb C(P) 
\end{equation} 
by \eqref{eq59}, \eqref{eq512}, and \eqref{eq519}. 
This is the desired surjection. 
\end{case}
\end{step}
Finally, we further assume that 
$P$ is a singular point of $X$. 
\begin{step} Note that 
$(X, \Delta)$ is klt in a neighborhood of $P$ by assumption. 
We will reduce the problem to the situation 
as in Step \ref{step4}. 
Let $\pi:Y\to X$ be the minimal resolution of $P$. 
We put $K_Y+\Delta_Y=\pi^*(K_X+\Delta)$. 
Since $\Bs|\pi^*D| =\pi^{-1}\Bs|D|$, 
it is sufficient to prove that 
$Q\not\in \Bs|\pi^*D|$ for some 
$Q\in \pi^{-1}(P)$. 
Since $\pi:Y\to X$ is the minimal resolution of $P$, 
$f:(Z, \Delta_Z)\to [X, K_X+\Delta]$ factors 
through 
$[Y, K_Y+\Delta_Y]$ and 
$(Z, \Delta_Z)\to [Y, K_Y+\Delta_Y]$ 
induces a natural quasi-log structure compatible with 
the original semi-log canonical structure 
of $(Y, \Delta_Y)$ (see Step \ref{step1} and \cite[Theorem 1.2]{fujino-slc}). 
We put $Y_0 =\pi^{-1}(X_0)$ where $P\in X_0$ as in Step \ref{step4}. 
We can take an effective $\mathbb R$-Cartier divisor 
$B'$ on $Y_0$ such that 
$B'\sim _{\mathbb R} (\pi|_{Y_0})^*A_0$, 
$\mult _Q B'>2$ for some $Q\in \pi^{-1}(P)$, and 
$B'=(\pi|_{Y_0})^*B$ for some effective $\mathbb R$-Cartier divisor $B$ on $X_0$. 
We put 
$U'=Y_0 \setminus \Sing Y_0$. 
We set  
\begin{equation}
c=\underset{t\geq 0}{\sup}\left\{t \left|
\begin{array}{l}  {\text{$(U', (\Delta_Y)|_{U'}+tB'|_{U'})$ is log canonical}}\\
{\text{at any point of $\pi^{-1}(P)$. }} 
\end{array}\right. \right\}. 
\end{equation} 
Then we have $0<c<1$. 
By adjunction (see Theorem \ref{f-thm4.4}) and 
Lemma \ref{f-lem4.6}, 
we can consider a quasi-log structure of $[Y_0, (K_Y+\Delta_Y)|_{Y_0}+cB']$. 
If there is a one-dimensional qlc center $C$ of 
$[Y_0, (K_Y+\Delta_Y)|_{Y_0}+cB']$ 
such that 
\begin{equation}
(\pi^*D-((K_Y+\Delta_Y)|_{Y_0}+cB'))\cdot C=(1-c)(\pi|_{Y_0})^*A_0\cdot C=0. 
\end{equation} 
Then we obtain that $C\subset \pi^{-1}(P)$. 
This means that $P$ is a qlc center 
of $[X_0, (K_X+\Delta)|_{X_0}+cB]$. 
In this case, we obtain surjections 
\begin{equation}
H^0(X, \mathcal O_X(D))\twoheadrightarrow
H^0(X_0, \mathcal O_{X_0}(D))\twoheadrightarrow 
\mathcal O_{X_0}(D)\otimes \mathbb C (P)
\end{equation} 
as in Case \ref{case1} in Step \ref{step4} (see \eqref{eq59} and \eqref{eq513}). 
Therefore, we may assume that 
\begin{equation}
(\pi^*D-((K_Y+\Delta_Y)|_{Y_0}+cB'))\cdot C>0
\end{equation}  
for every one-dimensional qlc center $C$ of $[Y_0, (K_Y+\Delta_Y)|_{Y_0}+cB']$. 
Note that 
\begin{equation}
(\pi^*D-(K_Y+\Delta_Y))\cdot C =(D-(K_X+\Delta))\cdot 
\pi_* C=A\cdot \pi_* C\geq 2 
\end{equation} 
when $\pi_*C\ne 0$, equivalently, $C$ is not a component of $\pi^{-1}(P)$. 
Then we can apply 
the arguments in Step \ref{step4} to 
$[Y_0, (K_Y+\Delta_Y)|_{Y_0}+cB']$ and $\pi^*D$. 
Thus, we obtain that $Q\not\in \Bs|\pi^*D|$ for some $Q\in \pi^{-1}(P)$. 
This means that $P\not\in \Bs|D|$.  
\end{step} 
Anyway, we obtain that $P\not\in \Bs|D|$. 
\end{proof}

By Theorem \ref{f-thm5.1}, 
we can quickly prove Corollary \ref{f-cor1.4} as follows. 

\begin{proof}[Proof of Corollary \ref{f-cor1.4}] 
We put $D=mI(K_X+\Delta)$ and $A=D-(K_X+\Delta)
= (m-1/I)I(K_X+\Delta)$. 
Then we obtain that $A\cdot C\geq m-1/I$ for 
every curve $C$ on $X$ and that 
$(A^2\cdot X_i)\geq (m-1/I)^2$ for every 
irreducible component $X_i$ of $X$. 
By Theorem \ref{f-thm5.1}, we obtain the desired freeness of 
$|mI(K_X+\Delta)|$. The very ampleness part follows from Lemma \ref{f-lem7.1} below. 
\end{proof}

\begin{rem}\label{f-rem5.3}
In Corollary \ref{f-cor1.4}, $\Delta$ is not necessarily reduced. 
If $\Delta$ is reduced, then Corollary \ref{f-cor1.4} 
is a special case of \cite[Theorem 24]{lr}. 
We note that $\Delta$ is always assumed to be reduced in \cite{lr}. 
\end{rem}

As a special case of Corollary \ref{f-cor1.4}, we can recover Kodaira's 
celebrated result (see \cite{kodaira}). 
We state it explicitly for the reader's convenience. 

\begin{cor}[Kodaira]\label{f-cor5.4}
Let $X$ be a smooth projective surface such that 
$K_X$ is nef and big. 
Then $|mK_X|$ is basepoint-free for every $m\geq 4$. 
\end{cor}

\begin{proof}[Proof of Corollary \ref{f-cor5.4}]
Apply Corollary \ref{f-cor1.4} to the canonical 
model of $X$. Then we obtain the desired freeness. 
\end{proof}

We close this section with the proof of Corollary \ref{f-cor1.5}. 

\begin{proof}[Proof of Corollary \ref{f-cor1.5}]
We put $D=-mI(K_X+\Delta)$ and $A=D-(K_X+\Delta)=-(m+1/I)I(K_X+\Delta)$. 
Then we obtain that $A\cdot C\geq m+1/I$ for every curve $C$ on $X$ and 
that $(A^2\cdot X_i)\geq (m+1/I)^2$ for 
every irreducible component $X_i$ of $X$. 
By Theorem \ref{f-thm5.1}, we obtain the desired freeness of 
$|-mI(K_X+\Delta)|$. 
The very ampleness part follows from Lemma \ref{f-lem7.1} below.
\end{proof}

\section{Log surfaces}\label{f-sec6} 
In this section, we prove Theorem \ref{f-thm1.6}. 
\begin{proof}[Proof of Theorem \ref{f-thm1.6}]
The proof is essentially the same as that of Theorem \ref{f-thm5.1}. 
However, there are some technical differences. 
We will have to use Theorem \ref{f-thm4.5} instead of 
Theorem \ref{f-thm4.4} (ii).  
So, we describe it for the reader's convenience. 
\setcounter{step}{0}
\begin{step}\label{step-s1} 
We take a resolution of singularities $f:Z\to X$ such that $\Supp f^{-1}_*\Delta
\cup \Exc(f)$ is a simple normal crossing divisor on $Z$, where 
$\Exc(f)$ is the exceptional locus of $f$. 
We put $K_Z+\Delta_Z=f^*(K_X+\Delta)$. 
Then, $(Z, \Delta_Z)$ gives a natural quasi-log structure on $[X, K_X+\Delta]$. 
\end{step}
\begin{step}\label{step-s2}
Assume that $(X, \Delta)$ is not log canonical at $x$. 
We put 
\begin{equation}
X'=\Nlc (X, \Delta)\cup \bigcup W, 
\end{equation} 
where $W$ runs over the one-dimensional 
log canonical centers of $(X, \Delta)$ such that 
$A\cdot W=0$. 
Then, by Theorem \ref{f-thm4.5}, 
we obtain 
\begin{equation}
H^i(X, \mathcal I_{X'}\otimes \mathcal O_X(D))=0
\end{equation} 
for every $i>0$, 
where $\mathcal I_{X'}$ is the defining ideal sheaf of $X'$. 
Note that $x$ is isolated in $\Supp \mathcal O_X/\mathcal I_{X'}$. 
Therefore, the restriction map 
\begin{equation}
H^0(X, \mathcal O_X(D))\to \mathcal O_X(D)\otimes \mathbb C(x)
\end{equation} 
is surjective. 
Thus, we obtain $x\not \in \Bs|D|$. 
\end{step}
From now on, we may assume that $(X, \Delta)$ is log canonical 
at $x$. 
\begin{step}
Assume that $x$ is a zero-dimensional 
log canonical center of $(X, \Delta)$. 
We put 
\begin{equation}
X'=\Nlc (X, \Delta)\cup \bigcup W\cup \{x\}, 
\end{equation} 
where $W$ runs over the one-dimensional 
log canonical centers of $(X, \Delta)$ such that 
$A\cdot W=0$. 
Then, by Theorem \ref{f-thm4.5}, 
we obtain 
\begin{equation}
H^i(X, \mathcal I_{X'}\otimes \mathcal O_X(D))=0
\end{equation} 
for every $i>0$. 
Note that $x$ is isolated in $\Supp \mathcal O_X/\mathcal I_{X'}$. 
Therefore, we obtain $x\not \in \Bs|D|$ as in 
Step \ref{step-s2}. 
\end{step}
From now on, we may assume that $(X, \Delta)$ is plt at $x$. 
\begin{step}\label{step-s4}
Assume that $(X, \Delta)$ is plt but is not klt at $x$. 
Let $L$ be the unique one-dimensional 
log canonical center of $(X, \Delta)$ passing through $x$. 
We put 
\begin{equation}
X' =\Nlc (X, \Delta) \cup \bigcup W \cup L
\end{equation} 
where $W$ runs over the one-dimensional log canonical 
centers of $(X, \Delta)$ such that $A\cdot W=0$. 
By Theorem \ref{f-thm4.5}, 
we obtain that  
\begin{equation}
H^i(X, \mathcal I_{X'}\otimes \mathcal O_X(D))=0 
\end{equation} 
for every $i>0$, as usual. 
Therefore, the restriction map 
\begin{equation}\label{eq68}
H^0(X, \mathcal O_X(D))\to H^0(X', \mathcal O_{X'}(D))
\end{equation} 
is surjective. 
By adjunction (see Theorem \ref{f-thm4.4}), 
$[X', (K_X+\Delta)|_{X'}]$ has a quasi-log structure 
induced by the quasi-log 
structure $f:(Z, \Delta_Z)\to [X, K_X+\Delta]$ constructed in Step \ref{step-s1}. 
Let $g:(Z', \Delta_{Z'})\to [X', (K_X+\Delta)|_{X'}]$ be the induced quasi-log resolution. 
We put 
\begin{equation}
c=\underset{t\geq 0}{\sup}\left\{t \left|
\begin{array}{l}  {\text{the normalization of 
$(Z', \Delta_{Z'}+tg^*x)$ is sub}}\\
{\text{log canonical over $X'\setminus \Nqlc((K_X+\Delta)|_{X'})$.}} 
\end{array}\right. \right\}. 
\end{equation} 
Then, by \cite[Lemma 3.16]{fujino-reid-fukuda}, we obtain that $0<c<2$. 
Note that $x$ is a Cartier divisor on $X'$. 
Let us consider $g:(Z', \Delta_{Z'}+cg^*x)\to [X', (K_X+\Delta)|_{X'}+cx]$, 
which defines a quasi-log structure. 
Then, by construction, $x$ is a qlc center of $[X', (K_X+\Delta)|_{X'}+cx]$. 
Moreover, 
we see that 
\begin{equation}
\deg(D|_L-(K_X+\Delta)|_L-cx)=(A\cdot L)-c>0
\end{equation} by assumption. 
We put 
\begin{equation}
X''=\Nqlc(X', (K_X+\Delta)|_{X'}+cx) \cup \bigcup W \cup \{x\}, 
\end{equation}  
where $W$ runs over the one-dimensional 
qlc centers of $[X', (K_X+\Delta)|_{X'}+cx]$ such that 
$W\ne L$. 
Then, by Theorem \ref{f-thm4.5}, we obtain 
\begin{equation}
H^i(X', \mathcal I_{X''}\otimes \mathcal O_{X'}(D))=0
\end{equation} 
for every $i>0$. 
Note that $x$ is isolated in $\Supp \mathcal O_{X'}/\mathcal I_{X''}$. 
Therefore, the restriction map 
\begin{equation}\label{eq613}
H^0(X', \mathcal O_{X'}(D))\to \mathcal O_{X'}(D)\otimes \mathbb C(x)
\end{equation} 
is surjective. 
By combining \eqref{eq68} with \eqref{eq613}, 
the desired restriction map 
\begin{equation}
H^0(X, \mathcal O_X(D))\to \mathcal O_X(D)\otimes \mathbb C(x)
\end{equation} 
is surjective. This means that 
$x\not\in \Bs|D|$. 
\end{step}
Thus, from now on, we may assume that $(X, \Delta)$ is klt at $x$. 
\begin{step}\label{step-s5}
In this step, we assume that $x$ is a smooth point of $X$. 
Since $A^2>4$, we can find an effective $\mathbb R$-Cartier divisor 
$B$ on $X$ such that 
$\mult_x B>2$ and that 
$B\sim _{\mathbb R} A$. 
We put 
\begin{equation}
c=\max\{ t\geq 0\, |\, (X, \Delta+tB) \ 
\text{is log canonical at $x$.}\}. 
\end{equation} 
Then we obtain that $0<c<1$ since $\mult _x B>2$. 
We have a natural quasi-log structure on $[X, K_X+\Delta+cB]$ as in 
Step \ref{step-s1}. 
By construction, 
there is a log canonical center of $[X, K_X+\Delta+cB]$ passing 
through $x$. 
We put 
\begin{equation}
X' =\Nlc (X, \Delta+cB)\cup \bigcup W \cup W_0, 
\end{equation}  
where $W_0$ is the minimal log canonical 
center of $(X, \Delta+cB)$ passing through $x$ 
and $W$ runs over the one-dimensional log 
canonical centers of $(X, \Delta+cB)$ such that 
$A\cdot W=0$. 
We note that $D-(K_X+\Delta+cB)\sim _{\mathbb R} (1-c)A$. 
Then, by Theorem \ref{f-thm4.5}, 
\begin{equation}\label{eq617}
H^i(X, \mathcal I_{X'}\otimes \mathcal O_X(D))=0
\end{equation}
for every $i>0$, where $\mathcal I_{X'}$ is the defining ideal sheaf of $X'$ on $X$. 
\setcounter{case}{0}
\begin{case}\label{case-s1}
If $\dim_x X'=0$, 
then $x$ is isolated in $\Supp \mathcal O_X /\mathcal I_{X'}$. 
Therefore, the restriction map 
\begin{equation}
H^0(X, \mathcal O_X(D))\to \mathcal O_X(D)\otimes \mathbb C(x)
\end{equation} 
is surjective. 
Thus, we obtain that $x\not\in \Bs|D|$. 
\end{case}
\begin{case}\label{case-s2}
If $\dim_x X'=1$, then $(X, \Delta+cB)$ is plt at $x$. 
We write $\Delta+cB=L+\Delta'$, where 
$L=W_0$ is the unique one-dimensional 
log canonical center of $(X, \Delta)$ passing through $x$ and $\Delta'=\Delta+cB-L$. 
We put 
\begin{equation}
c'=1-\mult _x (\Delta'|_L). 
\end{equation} 
Then $[X', (K_X+\Delta+cB)|_{X'}+c'x]$ has a quasi-log structure such that 
$x$ is a qlc center of 
this quasi-log structure as in Case \ref{case2} in Step \ref{step4} in the proof 
of Theorem \ref{f-thm5.1}. 
We put 
\begin{equation}
X''= \Nqlc (X', (K_X+\Delta+cB)|_{X'}+c'x) \cup \bigcup W \cup \{x\}, 
\end{equation} 
where $W$ runs over the one-dimensional qlc centers of 
$[X', (K_X+\Delta+cB)|_{X'}+c'x]$ such that $W\ne L$. 
By \eqref{eq517} in the proof of 
Theorem \ref{f-thm5.1}, we obtain that 
\begin{equation}\label{eq621}
\deg (D|_L-(K_X+\Delta+cB)|_L-c'x)>0. 
\end{equation} 
Then, by \eqref{eq621} and Theorem \ref{f-thm4.5}, 
\begin{equation}
H^i(X', \mathcal I_{X''}\otimes \mathcal O_{X'}(D))=0
\end{equation}  
for every $i>0$, where $\mathcal I_{X''}$ is 
the defining ideal sheaf of $X''$ on $X'$. 
Thus, we have that 
\begin{equation}\label{eq623}
H^0(X', \mathcal O_{X'} (D))\to \mathcal O_{X'}(D)\otimes 
\mathcal O_{X'} / \mathcal I_{X''}
\end{equation} 
is surjective. 
Note that $x$ is isolated in $\Supp \mathcal O_{X'}/ \mathcal I_{X''}$. 
Therefore, we obtain surjections  
\begin{equation}
\begin{split}
H^0(X, \mathcal O_X(D))\twoheadrightarrow 
H^0(X', \mathcal O_{X'}(D)) \twoheadrightarrow 
\mathcal O_{X'} (D)\otimes \mathbb C(x) 
\end{split}
\end{equation} 
by \eqref{eq617} and \eqref{eq623}. 
This is the desired surjection. 
\end{case}
\end{step}
Finally, we further assume that $x$ is a singular point of $X$. 
\begin{step}\label{step-s6}
Let $\pi:Y\to X$ be the minimal resolution of $x$. 
We put $K_Y+\Delta_Y=\pi^*(K_X+\Delta)$. 
Since $\Bs|\pi^*D| =\pi^{-1}\Bs|D|$, 
it is sufficient to prove that 
$y\not\in \Bs|\pi^*D|$ for some 
$y\in \pi^{-1}(x)$. 
Since $\pi:Y\to X$ is the minimal resolution of $x$,  
$f:(Z, \Delta_Z)\to [X, K_X+\Delta]$ factors 
through 
$[Y, K_Y+\Delta_Y]$ and 
$(Z, \Delta_Z)\to [Y, K_Y+\Delta_Y]$ induces a natural quasi-log structure on 
$[Y, K_Y+\Delta_Y]$. 
We can take an effective $\mathbb R$-Cartier divisor 
$B'$ on $Y$ such that 
$B'\sim _{\mathbb R} \pi^*A$, 
$\mult _y B'>2$ for some $y\in \pi^{-1}(x)$, and 
$B'=\pi^*B$ for some effective $\mathbb R$-Cartier divisor $B$ on $X$. 
We set  
\begin{equation}
c=\underset{t\geq 0}{\sup}\left\{t \left|
\begin{array}{l}  {\text{$(Y, \Delta_Y+tB')$ is log canonical}}\\
{\text{at any point of $\pi^{-1}(x)$.}} 
\end{array}\right. \right\}. 
\end{equation} 
Then we have $0<c<1$. 
As in Step \ref{step-s1}, 
we can consider a natural quasi-log structure of 
$[Y, K_Y+\Delta_Y+cB']$. 
If there is a one-dimensional qlc center $C$ of 
$[Y, K_Y+\Delta_Y+cB']$ 
such that $C\cap \pi^{-1}(x)\ne \emptyset$ and that 
\begin{equation}
(\pi^*D-(K_Y+\Delta_Y+cB'))\cdot C=(1-c)\pi^*A\cdot C=0. 
\end{equation} 
Then we obtain that $C\subset \pi^{-1}(x)$. 
This means that $x$ is a qlc center of $[X, K_X+\Delta+cB]$. 
In this case, we have that  
\begin{equation}
H^0(X, \mathcal O_X(D))\to \mathcal O_X(D)\otimes \mathbb C (x)
\end{equation} 
is surjective as in Case \ref{case-s1} in Step \ref{step-s5}. 
Therefore, we may assume that 
\begin{equation}
(\pi^*D-(K_Y+\Delta_Y+cB'))\cdot C>0
\end{equation}  
for every one-dimensional qlc center $C$ of $[Y, K_Y+\Delta_Y+cB']$ 
with $C\cap \pi^{-1}(x)\ne \emptyset$. 
We note that 
\begin{equation}
(\pi^*D-(K_Y+\Delta_Y))\cdot C =(D-(K_X+\Delta))\cdot \pi_* C 
= A\cdot \pi_* C\geq 2. 
\end{equation} 
Then we can apply 
the arguments in Step \ref{step-s5} to $[Y, K_Y+\Delta_Y+cB']$ and $\pi^*D$. 
Thus, we obtain that $y\not\in \Bs|\pi^*D|$ for some $y\in \pi^{-1}(x)$. 
This means that $x\not\in \Bs|D|$.  
\end{step}
Anyway, we obtain that $x\not\in \Bs|D|$. 
\end{proof}

\section{Effective very ampleness lemma}\label{f-sec7} 
In this section, we prove an effective very ampleness lemma. 
This section is independent of the other sections. 

The statement and the proof of 
\cite[1.2 Lemma]{kollar} do not seem to be true as stated. 
J\'anos Koll\'ar and the author think that 
we need some modifications. So, we 
prove the following lemma. 

\begin{lem}\label{f-lem7.1}
Let $(X, \Delta)$ be a projective semi-log canonical 
pair with $\dim X=n$. 
Let $D$ be an ample Cartier divisor on $X$ such that 
$|D|$ is basepoint-free. 
Assume that $L=D-(K_X+\Delta)$ is nef and log big with 
respect to $(X, \Delta)$, 
that is, 
$L$ is nef and $L|_W$ is big for every slc stratum $W$ of $(X, \Delta)$. 
Then $(n+1)D$ is very ample. 
\end{lem}

We give a detailed proof of Lemma \ref{f-lem7.1} for the reader's convenience. 

\begin{proof}
By the vanishing theorem (see \cite[Theorem 1.10]{fujino-slc}), 
we obtain that $H^i(X, \mathcal O_X((n+1-i)D))=0$ for every $i>0$. 
Then, by the Castelnuovo--Mumford regularity, we see that 
\begin{equation}
H^0(X, \mathcal O_X(D))\otimes H^0(X, \mathcal O_X(mD))
\to H^0(X, \mathcal O_X((m+1)D))
\end{equation} 
is surjective for every $m\geq n+1$ (see, 
for example, \cite[Chapter II.~Proposition 1]{kleiman}). 
Therefore, we obtain that 
\begin{equation}\label{eq7.2}
\mathrm{Sym}^kH^0(X, \mathcal O_X((n+1)D))\to H^0(X, \mathcal O_X(k(n+1)D))
\end{equation} 
is surjective for every $k\geq 1$. 
We put $A=(n+1)D$ and consider $f=\Phi_{|A|}: X\to Y$. Then 
there is a very ample Cartier divisor $H$ on $Y$ such that 
$A\sim f^*H$. 
By construction and the surjection 
\eqref{eq7.2}, we have the following commutative diagram 
\begin{equation}
\xymatrix{
\mathrm{Sym}^kH^0(Y, \mathcal O_Y(H)) \ar@{->>}[r]\ar[d]& 
\mathrm{Sym}^kH^0(X, \mathcal O_X(A))\ar@{->>}[d]\\
H^0(Y, \mathcal O_Y(kH))\ar@{^{(}->}[r]&H^0(X, \mathcal O_X(kA))
}
\end{equation}
for every $k\geq 1$. 
This implies that $H^0(Y, \mathcal O_Y(kH))\simeq H^0(X, \mathcal O_X(kA))$ for 
every $k\geq 1$. Note that $\mathcal O_Y\simeq f_*\mathcal O_X$ by 
\begin{equation}
0\to \mathcal O_Y \to f_*\mathcal O_X\to \delta\to 0 
\end{equation} 
and 
\begin{equation}
\begin{split}
0&
\to H^0(Y, \mathcal O_Y(kH))\to H^0(X, \mathcal O_X(kA))
\\ &
\to H^0(Y, 
\delta\otimes \mathcal O_Y(kH)) \to 
H^1(Y, \mathcal O_Y(kH))\to \cdots
\end{split} 
\end{equation} 
for $k\gg 0$. 
By the following commutative diagram: 
\begin{equation}
\xymatrix{
X\ar[r]^{f}\ar@{^{(}->}[dr]_{\Phi_{|kA|}}& Y\ar@{^{(}->}[d]^{\Phi_{|kH|}} \\
&\mathbb P^N, 
}
\end{equation}
where $k$ is a sufficiently large positive integer such that 
$kA$ and $kH$ are very ample, we obtain that 
$f$ is an isomorphism. 
This means that $A=(n+1)D$ is very ample. 
\end{proof}

We close this section with a remark on 
the very ampleness for 
$n$-dimensional stable pairs and 
semi-log canonical Fano varieties (see \cite{fujino-kollar-type}). 

\begin{rem}\label{f-rem7.2}
Let $(X, \Delta)$ be a projective semi-log canonical pair 
with $\dim X=n$. 

Assume that $I(K_X+\Delta)$ is an ample Cartier divisor for some 
positive integer $I$. 
Then we put $D=I(K_X+\Delta)$, $a=2$, and 
apply \cite[Remark 1.3 and Corollary 1.4]{fujino-kollar-type}. 
We obtain that $NI(K_X+\Delta)$ is very ample, where 
$N=(n+1)2^{n+1}(n+1)!(2+n)=2^{n+1}(n+2)!(n+1)$. 

Assume that $-I(K_X+\Delta)$ is an ample Cartier divisor for some 
positive integer $I$. 
Then we put $D=-I(K_X+\Delta)$, $a=1$, and 
apply \cite[Remark 1.3 and Corollary 1.4]{fujino-kollar-type}. 
We obtain that $-NI(K_X+\Delta)$ is very ample, where 
$N=(n+1)2^{n+1}(n+1)!(1+n)=2^{n+1}(n+1)^3n!$. 

Our results for surfaces in this paper are much sharper than the above 
estimates for $n=2$. 
\end{rem}

%%%%%%%%%%%%%%%%%%%%%%%%%%%

\end{document}